\documentclass[12pt,final]{article}
\usepackage{amsmath,amsthm,amsfonts,amssymb,graphicx,enumerate,psfrag}
\usepackage{tikz,pgfplots}
\usepackage{mathrsfs}
\usepackage{fullpage}

\usepackage[colorlinks,citecolor=blue,urlcolor=blue]{hyperref}
\definecolor{myblue}{RGB}{51,51,178}
\definecolor{myred}{RGB}{189,26,26}
\definecolor{mygreen}{RGB}{0,128,0}
\usepackage[utf8]{inputenc} 
\newtheorem{lemma}{Lemma}

\newtheorem{theorem}{Theorem}
\newtheorem{example}{Example}

\newtheorem{corollary}{Corollary}

\newcommand{\dN}{\mathbb {N}}

\newcommand{\I}{\mathcal{I}}

\newcommand{\cE}{\mathcal {E}}

\newcommand{\dR}{\mathbb {R}}
\newcommand{\dd}{\mathrm{d}}

\newcommand{\EE}{{\mathbb{E}}}
\newcommand{\PP}{{\mathbb{P}}}

\newcommand{\tmix}{\mathrm{t}_{\textsc{mix}}}
\newcommand{\wmix}{{\mathrm{w}_{\textsc{mix}}}}

\newcommand{\tv}{{\textsc{tv}}}
\newcommand{\Ent}{{\mathrm{Ent}}}

\newcommand{\Var}{{\mathrm{Var}}}

\newcommand{\Varent}{{\mathrm{Varent}}}

\title{Cutoff for non-negatively curved diffusions}
\author{Justin Salez}
\begin{document}
\maketitle
\begin{abstract}We resolve the long-standing problem of elucidating the cutoff phenomenon for a vast and important class of Markov processes, namely Markov diffusions with non-negative Bakry-\'Emery curvature. More precisely, we prove that any sequence of non-negatively curved  diffusions exhibits cutoff in total variation as soon as the product condition is satisfied. Our result holds in Euclidean spaces as well as on Riemannian manifolds, and for arbitrary non-random initial conditions. It vastly simplifies, unifies and generalizes a number of isolated works that have established cutoff through a delicate and model-dependent analysis of mixing times. The proof is elementary: we exploit a new simple differential relation between varentropy and  entropy to produce a quantitative bound on the width of the mixing window.
\end{abstract}

\section{Introduction}
Discovered four decades ago in the context of card shuffling \cite{aldous1986shuffling,aldous1983mixing,diaconis1996cutoff}, the cutoff phenomenon is an abrupt transition from out-of-equilibrium to equilibrium undergone by certain Markov processes: instead of decaying gradually over time, their total-variation distance to equilibrium remains close to the maximal value for a while and suddenly drops to zero as the time parameter reaches a critical threshold. Despite the accumulation of many examples, this phenomenon is
still far from being understood, and identifying the general conditions that trigger it has become one of the biggest challenges in the quantitative analysis of ergodic Markov processes. The present paper essentially solves this problem for a vast and important class of processes, namely non-negatively curved  diffusions. The latter occupy the center of the stage in the celebrated Bakry-\'Emery theory \cite{MR889476}, devoted to the analysis and geometry of Markov semi-groups.  We here only recall the necessary definitions, and refer the  reader to the excellent textbook \cite{MR3155209} for a comprehensive introduction to this beautiful framework.

\paragraph{Non-negatively curved diffusions.} Let $(E,\cE,\mu)$ be a probability space, and $(P_t)_{t\ge 0}$ a strongly continuous Markov semi-group of self-adjoint operators on $L^2(\mu)$. Its infinitesimal generator $L$ is defined by the formula
\begin{eqnarray*}
Lf & := & \lim_{t\to 0}\,\frac{P_tf-f}{t},
\end{eqnarray*}
for all $f\in L^2(\mu)$ such that the limit exists, and the  \emph{carré du champ} operator is given by
\begin{eqnarray*}
\Gamma (f,g) & := & \frac{1}{2}\left(L(fg)-fLg-gLf\right),
\end{eqnarray*}  
for all $f,g\in L^2(\mu)$ such that $f,g,fg\in\mathrm{Dom}(L)$. 
As usual, we simply write $\Gamma f$ when $f=g$. We make the following two important structural assumptions: 
\begin{enumerate}[(i)]
\item Chain rule: $\Gamma\left(\phi(f),g\right) = \phi'(f)\Gamma(f,g)$, for all  smooth functions $\phi\colon \dR\to\dR$. 
\item Non-negative curvature in the $\mathrm{CD}(0,\infty)$ sense: $\Gamma P_s g\le P_s\Gamma g$ for all $s\ge 0$.
\end{enumerate}
More precisely, we want those properties to hold for a rich enough class of functions $f,g\in L^2(\mu)$, in a sense that we now make clear. Consider  a Markov process $(X_t)_{t\ge 0}$ with the above generator, starting from an arbitrary point $x\in E$. We assume that for each $t>0$, the random variable $X_t$ admits a density $f_t$ with respect to $\mu$, and that both $f_t$ and $\log f_t$ are in $\mathrm{Dom}(L)$, as well as their squares and product. Moreover, we require that the conditions (i)-(ii) hold when $f=f_t$, $g=\log f_t$ and $\phi=\log$. Markov processes satisfying those requirements will henceforth be referred to as \emph{non-negatively curved diffusions}. An emblematic example is the Langevin diffusion in a convex potential. 
\begin{example}[Langevin diffusion in a convex potential]\label{ex:1} Let the $d-$dimensional Euclidean space $\dR^d$ be equipped with its Borel $\sigma-$algebra and a probability measure of the form
\begin{eqnarray}
\label{mu}
\mu(\dd x) & = & e^{-U(x)}\, \dd x,
\end{eqnarray}
where $U\colon\dR^d\to\dR$ is smooth and convex. Consider the stochastic differential equation
\begin{eqnarray}
\label{SDE}
\dd X_t & = & -\nabla U(X_t)\,\dd t+ \sqrt{2}\,\dd B_t,
\end{eqnarray}
where $(B_t)_{t\ge 0}$ is a standard $d-$dimensional Brownian motion. Then, the formula
\begin{eqnarray}
\label{def:Pt}
P_tf(x) & := & \EE\left[f(X_t)|X_0=x\right],
\end{eqnarray}
defines a Markov semi-group of self-adjoint operators on $L^2(\mu)$, whose generator and carré du champ operators act on smooth functions $f,g\colon\dR^d\to\dR$ as follows:
\begin{eqnarray}
\label{def:L}
Lf \ = \ \Delta f - \nabla U \cdot \nabla f, & \textrm{ and } & 
\Gamma(f,g) \ = \ \nabla f\cdot\nabla g.
\end{eqnarray}
In particular, the chain rule (i) trivially holds. Moreover, the non-negative curvature condition (ii) classically boils down to the convexity of $U$; see  \cite{MR889476,MR3155209} for details. 
\end{example}
More generally, the  Euclidean space $\dR^d$  can be replaced with any weighted Riemannian manifold, provided  the formulae (\ref{mu}-\ref{def:L}) are interpreted in an appropriate way. 
\begin{example}[Langevin diffusion on a weighted Riemannian manifold]\label{ex:2} Let  $M$ be  a Riemannian manifold equipped with a metric tensor $g$, and let $\mu$ be a probability measure on $M$ of the form (\ref{mu}), where $\dd x$ is  the volume measure on $M$ and $U\colon M\to\dR$ a geodesically semi-convex function. Now, consider the Langevin equation (\ref{SDE}), with $\nabla$ denoting the gradient in  $(M,g)$ and $(B_t)_{t\ge 0}$ a Brownian motion thereon.  Then, the formula (\ref{def:Pt})  defines a Markov semi-group of self-adjoint operators on $L^2(\mu)$ whose generator and carré du champ operators act as in (\ref{def:L}), with $\Delta$  denoting the Laplace-Beltrami operator on $(M,g)$. The chain rule (i) is again trivially satisfied, and the curvature condition  (ii) is known to holds as soon as 
\begin{eqnarray*}
 \mathrm{Ric}+\nabla^2U & \ge & 0,
 \end{eqnarray*}
where $\mathrm{Ric}$ denotes the Ricci curvature tensor. We again refer to \cite{MR889476,MR3155209} for details. 
\end{example}
\paragraph{Mixing times.} Consider a non-negatively curved diffusion $(X_t)_{t\ge 0}$ as above. Under mild assumptions, the law of $X_t$ will approach the equilibrium distribution $\mu$ as $t\to\infty$, and it is natural  to ask for the time-scale on which this convergence occurs. This is formalized by the notion of \emph{mixing times} \cite{MR3726904}, defined for any precision $\varepsilon\in(0,1)$ by
\begin{eqnarray*}
\tmix(\varepsilon) \ := \ \inf\{t>0\colon \tv(X_t)\le \varepsilon\}, & \textrm{where} &  \tv(X)  =  \sup_{A\in\cE}\left|\PP(X\in A)-\mu(A)\right|.
\end{eqnarray*}
Understanding how this quantity depends on the dynamics has been the subject of many works. Rather than asking \emph{how long} the process needs in order to approach equilibrium, we here   seek to understand \emph{how abrupt} the transition from out-of-equilibrium to equilibrium is. In other words, we ask for  the \emph{width} of the mixing window
\begin{eqnarray*}
\wmix(\varepsilon) & := & \tmix(\varepsilon)-\tmix(1-\varepsilon), 
\end{eqnarray*}
for $\varepsilon\in(0,1/2)$. The analysis of this second-order quantity is a notoriously difficult task, which has only been carried out in a handful of models with enough structure to allow for a more or less complete understanding of the  distance to equilibrium $t\mapsto \tv(X_t)$. 

\paragraph{Main result.} We provide a universal estimate on the width of the mixing window of any non-negatively curved diffusion, in terms of the \emph{spectral gap} of $L$.  Recall that the latter coincides with the largest  $\lambda\ge 0$ such that for any $f\in L^2(\mu)$ with $\mu(f)=0$ and any $t\ge 0$,
\begin{eqnarray}
\label{gap}
\|P_tf\|_2 & \le & e^{-\lambda t}\,\|f\|_2.
\end{eqnarray}
We assume that the spectral gap is positive, so that our statement below is not empty. 

\begin{theorem}[Mixing window]\label{th:main}The mixing window of any non-negatively curved diffusion starting from a deterministic point   satisfies 
\begin{eqnarray*}
\wmix(\varepsilon) & \le & \frac{3}{\lambda\varepsilon^3}+3\sqrt{\frac{\tmix(1-\varepsilon)}{\lambda\varepsilon^3}},
\end{eqnarray*}
for any $\varepsilon\in\left(0,\frac{1}{2}\right)$, where $\lambda$ is the spectral gap.
\end{theorem}
\paragraph{The cutoff phenomenon.} 
The interest of this estimate lies in the fact that it provides a very simple criterion for the occurrence of the  celebrated  \emph{cutoff phenomenon}. We refer the unfamiliar reader to the seminal works \cite{aldous1986shuffling,aldous1983mixing,diaconis1996cutoff} or the more recent paper \cite{MR4780485} for an introduction to this mysterious phase transition, illustrated on Figure \ref{fig:cutoff}. 
\begin{corollary}[Cutoff]\label{co:main}Consider a non-negatively curved diffusion whose state space, generator and initial state now depend on an implicit parameter $n\in\dN$, in such a way that the resulting spectral gap and mixing time satisfy 
\begin{eqnarray}
\label{PC}
\lambda\times \tmix(\varepsilon) & \xrightarrow[n\to\infty]{} & +\infty,
\end{eqnarray}
for some $\varepsilon\in(0,1)$. Then a cutoff occurs, in the sense that for all $\varepsilon\in(0,1)$,
\begin{eqnarray*}
\frac{\tmix(1-\varepsilon)}{\tmix(\varepsilon)} & \xrightarrow[n\to\infty]{} & 1.
\end{eqnarray*}
%In other words, the function $t\mapsto\tv(X_t)$ approaches a step function as $n\to\infty$.
\end{corollary}
\begin{figure}
\begin{center}
\pgfplotsset{compat=1.16}
\pgfplotsset{ticks=none}
\begin{tikzpicture}
	\begin{axis}[
		width = 14.5cm,
		height = 9cm,
		axis x line=middle,
		axis y line=middle,
		xlabel = $t$,
	    xlabel style ={at={(1,-0.1)}},
		ylabel style ={at={(0,1.05)}},
		ylabel = $\tv(X_t)$,
		clip = false,
		grid=both,
		grid style={dashed, line width=.5pt, draw=gray!10},
		xmode = normal,
		ymode = normal,
		line width = 1pt,
		legend cell align = left,
		legend style = {fill=none, at={(0.9,0.9)}, anchor = north east},
		yticklabel style={above left},
		anchor = north west,
		tickwidth={5pt},
		xtick align = outside,
		ytick align = outside,
		ymax = 1.05,
		ymin = 0,
		xmin = 0,
		xmax = 40,
		x axis line style=-,
		y axis line style=-,
		domain = 0:40,
		samples = 1000
		]
		\def\scale{3}
\addplot[solid,  myred, line width = 1.5 pt] {
-rad(atan(x-20))/rad(atan(20))/2 + 0.5	
	};
%\addplot[dashed, black] {1.00};
%
\draw[dashed, color= myblue] (15,0) -- (15,0.9515);
\draw[dashed, color= myblue] (0,0.9515) -- (15,0.9515);
\draw[dashed, color= mygreen] (25,0) -- (25,0.0484723);
\draw[dashed, color= mygreen] (0,0.0484723) -- (25,0.0484723);	

\draw[color= myblue] (0,0.9515) -- (-0.5,0.9515) node[color = black, anchor=east] {$\textcolor{myblue}{1-\varepsilon}$};
\draw[color= myblue] (15,0) -- (15,-0.02) node[color=black, anchor=north] {$\textcolor{myblue}{\tmix(1-\varepsilon)}$};
\draw[color= mygreen] (25,0) -- (25,-0.02) node[color = black, anchor=north] {${\textcolor{mygreen}{\tmix(\varepsilon)}}$};
\draw[color= mygreen] (0,0.0484723) -- (-0.5,0.0484723) node[color = black,anchor=east] {$\textcolor{mygreen}{\varepsilon}$};

\draw[color= black] (0,1.0) -- (-0.5,1.0) node[color = black,anchor=south east] {${1}$};

\end{axis}
	
\end{tikzpicture}
\caption{A typical plot of the distance to equilibrium $t\mapsto \tv(X_t)$. As the ratio $\frac{\tmix(1-\varepsilon)}{\tmix(\varepsilon)}$ approaches $1$, the transition to equilibrium becomes abrupt (cutoff).}
\label{fig:cutoff}
\end{center}
\end{figure}
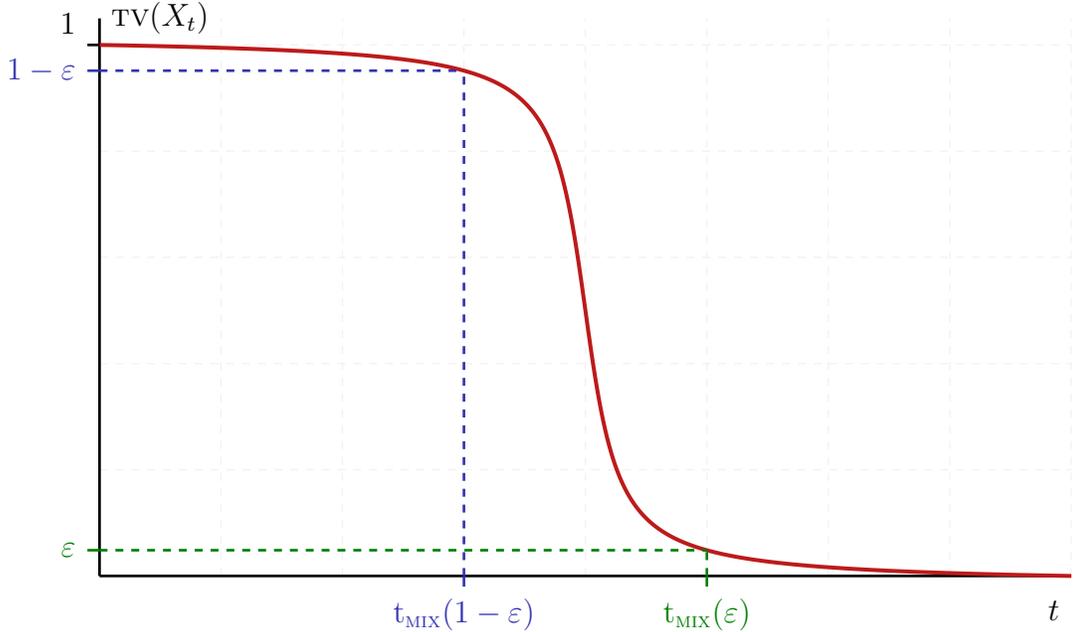
\paragraph{The product condition.} In the theory of mixing times, the assumption (\ref{PC}) is known as the \emph{product condition}. It was famously conjectured to imply cutoff for  all reversible Markov processes \cite{peresamerican}, and this has been verified on special ensembles such as birth-and-death chains \cite{ding2010total}, random walks on trees \cite{MR3650406}, or exclusion dynamics \cite{MR4546624}. However, the general conjecture was quickly disproved \cite[Section 6]{MR2375599}, and its failure is now known to be \emph{generic}, in the sense that the condition (\ref{PC}) is stable under a general rank-one perturbation which completely destroys cutoff; see  \cite[Example 18.7]{MR3726904}.  Nevertheless, by analogy with what is known for $L^p-$mixing when $p>1$ \cite{MR2375599},  the conjecture is still expected to hold for \emph{reasonable} processes. Our result confirms this long-standing prediction for non-negatively curved diffusions. 
\paragraph{Positive curvature.} Many examples of interest actually satisfy a stronger form of the sub-commutation property (ii), namely the $\mathrm{CD}(\kappa,\infty)$ criterion
\begin{eqnarray}
\label{cdk}
\forall t\ge 0,\qquad \Gamma P_t g & \le & e^{-2\kappa t}P_t\Gamma g,
\end{eqnarray}
for some $\kappa>0$ (see \cite{MR889476}). In particular, this holds in Example \ref{ex:1} as soon as $\nabla^2U\ge\kappa\mathrm{Id}$, and in Example \ref{ex:2} as soon as $\mathrm{Ric}+\nabla^2U\ge\kappa g$.  Under this positive curvature assumption, our main argument simplifies slightly and yields the following neater estimate.
\begin{theorem}[Mixing window under positive curvature]\label{th:positive}The mixing window of any positively curved diffusion starting from a deterministic point   satisfies 
\begin{eqnarray}
\label{window}
\wmix(\varepsilon) & \le & \frac{3}{\kappa\varepsilon^2},
\end{eqnarray}
for any $\varepsilon\in\left(0,\frac{1}{2}\right)$, where $\kappa$ is the curvature. In particular, cutoff occurs as soon as 
\begin{eqnarray}
\label{PC2}
\kappa \times \tmix(\varepsilon) & \xrightarrow[n\to\infty]{} & +\infty.
\end{eqnarray}
\end{theorem}
Since we classically have $\lambda\ge\kappa$, the criterion (\ref{PC2}) does not improve upon the product condition. However, the estimate on the width of the mixing window (\ref{window}) can be significantly sharper than the one provided by Theorem \ref{th:main}.

\paragraph{Worst-case mixing times.} In the literature on the cutoff phenomenon, it is fairly standard to consider \emph{worst-case} mixing times, of the form
\begin{eqnarray*}
\tmix^{(S)}(\varepsilon) & := & \max_{x\in S}\tmix^{(x)}(\varepsilon),
\end{eqnarray*}
where $S\subseteq E$ is a particular set of initial configurations, and where we have written   $\tmix^{(x)}(\varepsilon)$ to indicate the dependency on the initial state. Our result is strong enough to cover this extension as well. Indeed, Theorem \ref{th:main} asserts that for all $x\in E$, 
\begin{eqnarray*}
\tmix^{(x)}(\varepsilon) & \le & \tmix^{(x)}(1-\varepsilon)  + \frac{3}{\lambda\varepsilon^3}+3\sqrt{\frac{\tmix^{(x)}(1-\varepsilon)}{\lambda\varepsilon^3}},
\end{eqnarray*}
so we may take a maximum over $x\in S$ to deduce that the same statement holds with  $\tmix^{(S)}$ instead of $\tmix^{(x)}$. This remark of course also applies to Corollary \ref{co:main} and Theorem \ref{th:positive}. On compact manifolds, the standard choice is $S=E$, in which case the product condition is actually well known to be necessary for cutoff. Thus, Corollary \ref{co:main} completely characterizes cutoff for non-negatively curved diffusions on compact manifolds.

\paragraph{Special cases.}  Because it only requires rough lower bounds on the spectral gap and the mixing time, the product condition is easily verified in many models of interest. As a consequence, Corollary \ref{co:main} vastly simplifies, unifies and generalizes a number of isolated proofs of cutoff that were obtained through an in-depth, model-specific analysis of the underlying semi-group. Some emblematic examples include Brownian motion on high-dimensional spheres \cite{MR1306030}, Ornstein-Uhlenbeck processes \cite{MR3770869}, the Dyson-Ornstein-Uhlenbeck process \cite{MR4528974},  as well as certain spectrally rigid diffusions \cite{chafaï2024cutoffrigidityhighdimensional}.

\section{Proof}
Following the ideas exposed at length in \cite{MR4780485,MR4765357}, we let the key role be played by an information-theoretic notion known as \emph{varentropy}, which we now recall. Let $X$ be a random variable taking values in our workspace $(E,\cE)$, and admitting a density $f$ with respect to the reference measure $\mu$. The \emph{information content} of $X$ is defined as  
\begin{eqnarray*}
\I_X & := & \log f(X).
\end{eqnarray*}
The usual definition has a minus sign \cite{MR2239987}, but we find this formulation more convenient since the expectation of $\I_X$ is then exactly the relative entropy of $X$ with respect to $\mu$:
\begin{eqnarray*}
\Ent(X) & := & \EE[\I_X] \ = \ \int_E f\log f\,\dd \mu.
\end{eqnarray*} 
The varentropy of $X$ is obtained by replacing the above expectation with a variance:
\begin{eqnarray*}
\Varent(X) & := & \Var(\I_X) \ = \ \int_E f(\log f)^2\,\dd \mu - \left(\int_E f\log f\,\dd \mu\right)^2.
\end{eqnarray*}
This natural statistics controls  the fluctuations of the information content around the entropy. It appeared a decade ago in the context of optimal data compression \cite{inproceedings}, but has since then been shown to play an important role in  completely different areas, such as importance sampling \cite{MR3784496} or the cutoff phenomenon \cite{MR4780485,MR4765357,hermon2024concentrationinformationdiscretegroups}. Its main interest here lies in the fact that it allows one to reverse the celebrated Pinsker inequality, i.e., to obtain an upper-bound on the relative entropy in terms of the total-variation distance. This is the content of the following  lemma, borrowed from  \cite[Lemma 8]{MR4780485}.
\begin{lemma}[Reverse Pinsker's inequality]\label{lm:pinsker}We always have
\begin{eqnarray}
\Ent(X) & \le & \frac{1+\sqrt{\Varent(X)}}{1-\tv(X)}.
\end{eqnarray}
\end{lemma}
We also borrow from the same work an estimate on the mixing time of any reversible Markov process, based on its spectral gap and initial entropy; see \cite[Lemma 7]{MR4780485}.
\begin{lemma}[Mixing-time estimate]\label{lm:gap}Let $(X_t)_{t\ge 0}$ be a reversible Markov process with spectral gap $\lambda$, starting from a (possibly random) initial condition $X_0$. Then, 
\begin{eqnarray*}
\forall \varepsilon\in(0,1),\qquad \tmix(\varepsilon) & \le & \frac{1+\Ent(X_0)}{\lambda\varepsilon}.
\end{eqnarray*}
\end{lemma}
Note that when $X_0$ is deterministic and $\mu$ is diffuse (which is the case in our examples), this result seems  useless since the right-hand side is infinite. However, by the Markov property, the conclusion automatically ``propagates'' in time as follows: for any $t\ge 0$,
\begin{eqnarray}
\label{gap:t}
\tmix(\varepsilon) & \le & t+\frac{1+\Ent(X_t)}{\lambda\varepsilon}.
\end{eqnarray}
In particular, we may take $t=\tmix(1-\varepsilon)$ and combine this with Lemma \ref{lm:pinsker} to arrive at
\begin{eqnarray*}
\wmix(\varepsilon) & \le & \frac{2}{\lambda\varepsilon^2}+\frac{1}{\lambda\varepsilon^2}\sqrt{\Varent(X_t)}. 
\end{eqnarray*}
This general estimate on the width of the mixing window leads to the \emph{varentropy criterion} for cutoff discovered in \cite{MR4780485}. Exploiting the latter in the present setting would require us to estimate  the varentropy of non-negatively curved diffusions at the mixing time, in a similar spirit as what was done in the discrete setting in \cite{MR4780485,hermon2024concentrationinformationdiscretegroups}. This is, however, \emph{not} the approach that we pursue here. Instead, we exploit  the  new crucial observation that for non-negatively curved diffusions, entropy and varentropy are related through a simple differential inequality. 
\begin{lemma}[Information-theoretic differential inequality]\label{lm:diff}
If $(X_t)_{t\ge 0}$ is a non-negatively curved diffusion starting from a deterministic state, then for all $t>0$ 
\begin{eqnarray*}
\frac{\dd}{\dd t}\,\Ent(X_t) & \le & -\frac{\Varent(X_t)}{2t}.
\end{eqnarray*}
\end{lemma}
\begin{proof}
The sub-commutation property $\Gamma P_t g\le P_t\Gamma g$ for $t\ge 0$ stated at (ii) easily and classically guarantees the local Poincaré inequality
\begin{eqnarray}
\label{local}
\Var\left[g(X_t)\right] & \le & 2t\,\EE\left[\Gamma g(X_t)\right],
\end{eqnarray}
for all $t\ge 0$. Choosing $g=\log f_t$ and using the chain rule (i) with $\phi=\log$, we obtain 
\begin{eqnarray*}
\Varent(X_t) & \le & 2t\,\EE\left[\frac{\Gamma(f_t,\log f_t)(X_t)}{f_t(X_t)}\right].
\end{eqnarray*}
On the other hand, using reversibility and the fact that $\frac{\dd}{\dd t}f_t=Lf_t$, we have
\begin{eqnarray*}
-\frac{\dd}{\dd t}\,\Ent(X_t) & = & \int_{E}{\Gamma(f_t,\log f_t)}\, \dd\mu  \ = \ 
\EE\left[\frac{\Gamma(f_t,\log f_t)(X_t)}{f_t(X_t)}\right].
\end{eqnarray*}
The claim readily follows. 
\end{proof}
We now have everything we need to prove our main result. 
\begin{proof}[Proof of Theorem \ref{th:main}]Fix $\varepsilon\in(0,\frac 12)$ and set $t_0:=\tmix(1-\varepsilon)$. Combining Lemmas  \ref{lm:pinsker} and \ref{lm:diff}, we see that on $[t_0,\infty)$, the function $t\mapsto \Ent(X_t)$ obeys the differential inequality
\begin{eqnarray}
\label{diff}
\frac{\dd}{\dd t}\,\Ent(X_t) & \le & -\frac{\left(\varepsilon\Ent(X_t)-1\right)^2}{2t}.
\end{eqnarray}
Integrating this inequality, we obtain
\begin{eqnarray*}
\frac{1}{\varepsilon\Ent(X_t)-1} & \ge & \frac{1}{\varepsilon\Ent(X_{t_0})-1}+\frac{\varepsilon}{2}\log\frac{t}{t_0},
\end{eqnarray*}
provided $t> t_0$ and $\Ent(X_t)>\frac{1}{\varepsilon}$. Consequently, for all $t> t_0$,
\begin{eqnarray*}
\Ent(X_t) & \le & \frac{1}{\varepsilon}+\frac{2}{\varepsilon^2\log\frac{t}{t_0}}.
\end{eqnarray*}
Invoking Lemma \ref{lm:gap}, or rather its propagated version (\ref{gap:t}), we conclude that
\begin{eqnarray*}
\tmix(\varepsilon) & \le & t+\frac{1}{\lambda\varepsilon}+\frac{1}{\lambda\varepsilon^2}+\frac{2}{\lambda\varepsilon^3\log\frac{t}{t_0}}\\
& \le & t+\frac{3}{\lambda\varepsilon^3}+\frac{2t_0}{\lambda\varepsilon^3(t-t_0)},
\end{eqnarray*}
where the second line uses $\log u\ge 1-\frac{1}{u}$ and $\varepsilon\le 1/2$. Now, this bound is valid for any $t>t_0$, so we may finally optimize it by choosing $t=t_0+\sqrt{\frac{2t_0}{\lambda\varepsilon^3}}$ to obtain
\begin{eqnarray*}
\tmix(\varepsilon) & \le & t_0+\frac{3}{\lambda\varepsilon^3}+\sqrt{\frac{8t_0}{\lambda\varepsilon^3}}. 
\end{eqnarray*}
Recalling that $t_0=\tmix(1-\varepsilon)$, this is enough to conclude. 
\end{proof}
\begin{proof}[Proof of Theorem \ref{th:positive}]The stronger condition (\ref{cdk}) implies the time-uniform Poincaré inequality
\begin{eqnarray*}
\Var\left[g(X_t)\right] & \le & \frac{1}{\kappa}\,\EE\left[\Gamma g(X_t)\right],
 \end{eqnarray*}
instead of (\ref{local}). Thus, the differential inequality (\ref{diff}) in the above proof simplifies to
\begin{eqnarray*}
\frac{\dd}{\dd t}\,\Ent(X_t) & \le & -\kappa\left(\varepsilon\Ent(X_t)-1\right)^2.
\end{eqnarray*}
Integrating leads to the bound
\begin{eqnarray*}
\Ent(X_t) & \le & \frac{1}{\varepsilon}+\frac{1}{\varepsilon^2\kappa(t-t_0)},
\end{eqnarray*}
so that (\ref{gap:t}) along with the estimate $\lambda\ge\kappa$ now yields
\begin{eqnarray*}
\tmix(\varepsilon) & \le & t+\frac{1}{\kappa\varepsilon}+\frac{1}{\kappa\varepsilon^2}+\frac{1}{\kappa^2\varepsilon^3(t-t_0)}.
\end{eqnarray*}
This is valid for any $t> t_0$, and choosing $t=t_0+\kappa^{-1}\varepsilon^{-3/2}$ concludes the proof.
\end{proof}

\section*{Acknowledgment.}This work is supported by the ERC consolidator grant CUTOFF (101123174). Views and opinions expressed are however those of the authors only and do not necessarily reflect those of the European Union or the European Research Council Executive Agency. Neither the European Union nor the granting authority can be held responsible for them.

\bibliographystyle{plain}
\bibliography{draft}

\end{document}